\RequirePackage{fix-cm}
\documentclass[smallextended]{svjour3}       
\smartqed  
\usepackage{graphicx}

\usepackage{subfigure}
\usepackage[margin=1in]{geometry}  
\usepackage{amsmath}               
\usepackage{amsfonts}              
\usepackage{amsthm}                
\usepackage[utf8]{inputenc}
\usepackage{hyperref}
\usepackage{easybmat}
\usepackage{setspace}
\begin{document}

\title{Evolving Affine Evolutoids}



\author{Ady Cambraia Junior         \and
        Abílio Lemos 
}

\institute{Ady Cambraia Junior,  Abílio Lemos  \at
              Departamento de Matemática, Universidade Federal de Viçosa, Av. P. H. Rolfs, s/n, Campus Universitário, CEP: 36570-900, Viçosa, Minas Gerais, Brazil. \\
              \email{ady.cambraia@ufv.br, abiliolemos@ufv.br}  \           
}

\date{Received: date / Accepted: date}

\maketitle

\begin{abstract}

The envelope of straight lines affine normal to a plane curve $C$ is its affine evolute; the envelope of the affine lines tangent to $C$ is the original curve, together with the entire affine tangent line at each inflexion of $C$. In this paper, we consider plane curves without inflexions. We use some techniques of singularity theory to explain how the first envelope turns into the second, as the (constant) slope between the set of lines forming the envelope and the set of affine tangents to $C$ changes from $0$ to $1$. In particular, we guarantee the existence of the first slope for which singularities occur. Moreover, we explain how these singularities evolve in the discriminant surface.

\end{abstract}

\section{Introduction}
\label{intro}

Let $\gamma$ be a plane curve, which we shall assume closed, smooth and without affine inflexions. The {\it envelope} of a family of lines is formed by intersections of infinitesimal consecutive lines or equivalently a curve tangent to all the lines. For example, the envelope of the family of affine tangents lines to $\gamma$ contains at least itself and the envelope of affine normals is called the {\it affine evolute} of $\gamma$.

It is natural to ask what lies "between" the envelope of affine tangents and the envelope of affine normals. Let us fix a number $\alpha$ ranging between 0 and 1 and consider the lines $L_\alpha$ that through by $\gamma(s)$ of slope $\alpha\gamma_s+(1-\alpha)\gamma_{ss}$, where $s$ is the parameter of affine arc-length. The euclidean case was investigated by Giblin and Warder \cite{Giblin3}.

This work explicits the envelope of lines $L_\alpha$, which we call affine evolutoid, and provide some results, such as: the regularity conditions of the envelope, existence of first $\alpha$ where singularities and conditions for existence of ordinary affine cusps occur. Moreover, we apply the results of the singularity theory to prove how the singularities evolve on the discriminant of the family to three parameters obtained from the equations that define $L_\alpha.$  More precisely, we found (locally) that the discriminant surfaces are cuspidal edges or swallowtail surfaces.

\section{Review of the affine geometry of planar curves}


In this section, we present the basic concepts of the affine differential geometry of planar smooth curves. For further details, see \cite{Nomizu,Buchin}.	

Let $\gamma: [0,1] \longrightarrow \mathbb{R}^2$ be a planar curve parametrized by $t$. The basic purpose of the planar affine differential geometry is to define a new parametrization, $s$, which is an affine-invariant, and the simplest affine-invariant parametrization $s$ is given by requiring, at every curve point $\gamma(s)$, the relation 
\begin{equation}\label{pcaa}
[ \gamma_s, \gamma_{ss}]=1,
\end{equation}

\noindent where $[ ,]$ is the notation for determinants. 
When a curve satisfies equation \eqref{pcaa}, we say it is
parameterized by affine arclength.

The vectors $\gamma_s$ and $\gamma_{ss}$ are the affine tangent and the affine normal, respectively.

The parameters $s$ and $t$ are related by 
$$ [\gamma_t,\gamma_{tt}]=\left[\gamma_ss_t,
\gamma_{ss}(s_t)^2+\gamma_ss_{tt}\right]=s_t^3[\gamma_s,\gamma_{ss}]=s_t^3$$
Thus,
$$ \dfrac{ds}{dt}= [\gamma_t,\gamma_{tt}]^{\frac{1}{3}}.$$
By differentiating the equation \eqref{pcaa}, we obtain
$$ [\gamma_s, \gamma_{sss}]=0 \Rightarrow
\gamma_{sss}+ \mu(s)\gamma_s=0,$$ for some $\mu(s) \in \mathbb{R}.$
The function $\mu(s)$ is the affine curvature and the simplest non-trivial affine differential invariant. Notice that
$$ [\gamma_{s}, \gamma_{sss}]=0  \Rightarrow \gamma_{sss}=-\mu(s)\gamma_s,$$ therefore, we conclude that
$$\mu(s)=[\gamma_{ss}, \gamma_{sss}].$$

\begin{theorem}\cite{Nomizu} \label{thm1}
Curves have constant affine curvature if and only if they are conic sections.
\end{theorem}

\section{The affine normal and the affine curvature of a curve non parameterized by affine arclenght}

\begin{proposition} Let $\gamma: \mathbb{R} \longrightarrow \mathbb{R}$ be a regular curve parametrized by an arbitrary parameter $t$. The affine normal $\xi(t)$ is given by:

\begin{eqnarray*}\label{eq0}
 \xi(t)=\kappa^{-\dfrac{2}{3}}\gamma_{tt}-\dfrac{1}{3}\kappa_t\kappa^{-\dfrac{5}{3}}\gamma_t
\end{eqnarray*}
\end{proposition}
%
%
%
%
%

The affine curvature of a planar curve $\gamma$ parametrized by an arbitrary parameter is given in the next result.

\begin{proposition}
Let $\gamma$ be a smooth plane curve without inflexion points parametrized by an arbitrary parameter $t$. Considering
$\kappa=[\gamma_t,\gamma_{tt}]$, we conclude that the affine curvature is given by

\begin{eqnarray} \label{eq00}
\mu = \dfrac{1}{9}\left(3\kappa\kappa_{tt}-5\kappa_t^2+9\kappa[\gamma_{tt},\gamma_{ttt}]\right)\kappa^{-\frac{8}{3}}.
\end{eqnarray}
\end{proposition}

\begin{proof}
Note that $s_t=\kappa^{\frac{1}{3}}$ and
$\gamma_s=\gamma_t\kappa^{-\frac{1}{3}}$. Now, calculate
$\gamma_{ss}$, $\gamma_{sss}$ and use the fact that
$\kappa_t=[\gamma_t,\gamma_{ttt}]$, thus
$\mu=[\gamma_{ss},\gamma_{sss}].$
\end{proof}
Consider a plane curve in the Monge's form without euclidean inflexions close to origin, that is,

$$\gamma(t)=\left(t, \displaystyle\frac{1}{2}a_2t^2+ \cdots +\dfrac{1}{k!}a_kt^k+g(t)t^{k+1}
\right),$$ where $a_i \in \mathbb{R}$, $a_2\neq0$ and $g$ is a smooth function. Using the previous theorem, the affine curvature of $\gamma$ in
$\gamma(0)$ is

$$\mu(0)=\dfrac{3a_2a_4-5a_3^2}{9a_2^{\frac{8}{3}}}$$

This means that the affine curvature function is an invariant affine differential of order 4 of $\gamma$.

\section{Affine Envelopes}\label{sec4}


Let $\gamma: I \longrightarrow \mathbb{R}^2$ be a smooth closed curve without affine inflexions. It is known that the envelope of affine tangents to $\gamma$ is formed by the curve itself and by affine tangents in the affine inflexion points, \cite{Sano}. It is also known that the affine normals are the affine evolute of curve $\gamma$. Inspired in the work \cite{Giblin3}, we asked what the envelope of lines with slope between affine tangent and affine normal to curve $\gamma$ would be. 

Let $(1-|\alpha|)\gamma_s+\alpha\gamma_{ss}$ be a vector between $\gamma_s$ and $\gamma_{ss}$, where $\alpha \in [-1,1]$. In this paper, we consider the case where $\alpha>0$, the case $\alpha<0$ is similar. 

We are interested in the envelope of lines with slope $v^\alpha=(1-\alpha)\gamma_s+\alpha\gamma_{ss}, \alpha \in [0,1],$ which we denote by $L_\alpha$. The equation of line $L_\alpha$ is given by $$\begin{array}{cccl}
F: & \mathbb{R}^2\times I & \longrightarrow & \mathbb{R}^2 \\
  & (X,s) & \longmapsto & F(X,s)=\left[X-\gamma, (1-\alpha)\gamma_s+\alpha\gamma_{ss}\right] \\
\end{array},
$$ where $[ \, ,]$ is the notation for determinants.

For $\alpha$ fixed, $F(X,s)=0$ refers to a family of lines, e.g., for each $\alpha$ we have a line and when $s$ varies, the line moves in the plane $xy.$

The envelope of family $F(X,s)$ is given by $$E_\alpha=\left\{X=(x,y)\in \mathbb{R}^2 | \textrm{there is} \ s \ \textrm{such that} \ F(X,s)=F_s(X,s)=0\right\}.$$ 

As $\alpha$ is fixed (constant) , it follows that $$F_s(X,s)=\left[-\gamma_s,(1-\alpha)\gamma_s+\alpha\gamma_{ss}\right]+\left[X-\gamma,(1-\alpha)\gamma_{ss}-\alpha\mu\gamma_s\right].$$

Here, we use the fact that $s$ is the parameter affine arclenght. Therefore, $\gamma_{sss}=-\mu(s)\gamma_s,$ where $\mu$ is the affine curvature of $\gamma$.
By solving the system $F=F_s=0$, we obtain

\begin{equation}\label{envelope}
X(s)=\gamma(s) + \dfrac{\alpha}{(1-\alpha)^2+\mu(s)\alpha^2}\left((1-\alpha)\gamma_s(s)+\alpha\gamma_{ss}(s)\right).
\end{equation} 

\begin{remark} \label{r1}
\
\begin{enumerate}
\item[(a)] Notice that $(1-\alpha)^2+\mu(s)\alpha^2\neq0.$ Otherwise, the affine curvature should be a negative constant and thus $\gamma$ would not be closed, see Theorem \ref{thm1}. 
\item[(b)] If $\alpha=1$, then the lines $F(X,s)=0$ are the affine normals to $\gamma$ and the envelope is the affine evolute, e. g., the set of points $\gamma+\dfrac{1}{\mu}\gamma_{ss}$, which are centers of conics doing $5-$contact with $\gamma$, also called centers of affine curvature of $\gamma.$ 
\item[(c)] If $\alpha=0$, then the lines are the affine tangents to $\gamma$ and the envelope is the original curve $\gamma.$ 
\end{enumerate}
\end{remark}


\section{Regularity of envelope}


Consider the envelope of family $F$ given by equation \eqref{envelope}. We propose to investigate when this curve is regular or not regular. In the next proposition, we give the conditions for this.

%

\begin{proposition}\label{regularidade_do_envelope}
The envelope \eqref{envelope} is not regular if and only if 
\begin{equation}\label{nonregularity}
\mu_s= \dfrac{(1-\alpha)((1-\alpha)^2+\mu\alpha^2)}{\alpha^3},
\end{equation} where $\mu_s$ is the derivative of affine curvature with respect to $s$ on $\gamma$.
\end{proposition}

\begin{proof}
Assume $\mu(s)\neq0$. By differentiating  the solution \eqref{envelope} of the envelope of $F$ with respect to parameter affine arclength $s$, we obtain $X_s=A(s)\left((1-\alpha)\gamma_s+\alpha\gamma_{ss}\right),$ where $$A(s)=\dfrac{(1-\alpha)}{{(1-\alpha)^2+\mu\alpha^2}}-\dfrac{\alpha^3\mu_s}{[(1-\alpha)^2+\mu\alpha^2]^2}.$$
Therefore, $X_s$ is zero if and only if $A(s)=0$, e. g., $\mu_s= \dfrac{(1-\alpha)((1-\alpha)^2+\mu\alpha^2)}{\alpha^3}.$
\end{proof}

For $\alpha=1$, the envelope corresponds to affine evolute. By differentiating the equation \eqref{envelope}, we obtain the familiar condition $\mu_s=0$, e. g., it says that $\gamma$ has an extreme of affine curvature, e. g., $\gamma$ has an affine vertex. In the case $\alpha=0$, the envelope corresponds to curve $\gamma$ itself, which is regular by assumption.


\begin{example}\label{ellipse}
Consider an ellipse parameterized by $\gamma(t)=(a\cos(t),b\sin(t)),$ where $b>a>0$ (for $a=2, b=3$ \ see Fig. \ref{fig1}). The reparameterization by affine arclenght is $\alpha(s)=\left(a\cos\left(\dfrac{s}{(ab)^{\frac{1}{3}}}\right),b\sin\left(\dfrac{s}{(ab)^{\frac{1}{3}}}\right)\right)$. If we apply the condition of Proposition \ref{regularidade_do_envelope}, we conclude that, for any $\alpha$, the affine evolutoid is smooth. This was expected because the affine curvature of $\alpha$ is always constant.
\end{example}

\begin{figure}[ht]
  \centering
  \includegraphics[scale=0.2]{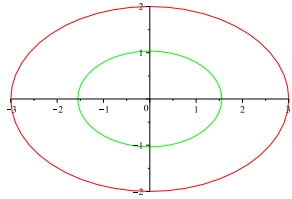}\\
  \caption{Ellipse $\gamma(t)=(2\cos(t),3\sin(t))$ and the affine evolutoid to $\alpha=0.75$. In true, for all $0\leq\alpha<1$, the affine evolutoids are smooth and for $\alpha=1$ the affine evolutoid is the degenerated affine evolute.}
  \label{fig1}
\end{figure}
%
\begin{example}\label{ex2}
Consider the curve $\gamma(t)=\left(\cos(2t)-\cos(t+a), \sin(2t)+\sin(t)\right)$. Here, the affine evolutoid presents singularities (see Fig. \ref{fig2}).
\end{example}
\bigskip

\begin{figure}[ht]
  \centering
  \includegraphics[scale=0.2]{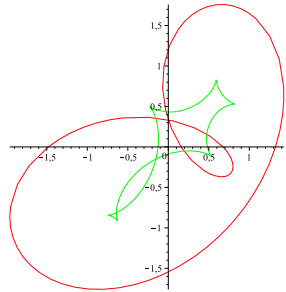}\\
  \caption{Curve $\gamma(t)=\left(\cos(2t)-\cos(t+1.9), \sin(2t)+\sin(t)\right)$ and the affine evolutoid for $\alpha=0.9$. }
  \label{fig2}
\end{figure}

The existence of a first $\alpha$ such that the affine evolutoid is not smooth is guaranteed in the next result.

\begin{theorem}[$\alpha$ born]\label{born}
Consider $\gamma$ as in Section \ref{sec4}. There is a first $\alpha$ such that the condition \eqref{nonregularity} given in Proposition \eqref{regularidade_do_envelope} occurs.
\end{theorem}
\begin{proof}
The ordinary differential equation given in the condition \eqref{nonregularity} has the solution below: $$\mu(s)=-\dfrac{(1-\alpha)^2}{\alpha^2}+ Ce^{\dfrac{(1-\alpha)}{\alpha}s},$$ where $C\in \mathbb{R}.$
Define the function $G : (0,1]\times I  \longrightarrow   \mathbb{R},$ given by $G(\alpha,s)=-\mu(s)-\dfrac{(1-\alpha)^2}{\alpha^2}+ Ce^{\dfrac{(1-\alpha)}{\alpha}s}$. Fixing $s$, such function is continuous and defines a curve. Observe that $\lim_{\alpha\rightarrow0^+}|G(\alpha,s)|=\infty.$ Then, given any real number $M>0$, there is $\delta>0$, such that $|G(\alpha,s)|>M.$ Thus, $G|_{[\delta,1]\times\{s\}}$ is a continuous function defined in a compact set. Therefore there is $\alpha_0=\alpha_0(s)$, such that $G(\alpha_0,s)=\min_{\delta\leq\alpha\leq1}(G(\alpha,s))$.
\end{proof}

\begin{remark}
In the Example \ref{ellipse}, the $\alpha$ born occurs for $\alpha=1$. On the other hand, in the Example \ref{ex2}, the existence of first $\alpha$ such that the affine evolutoid is not smooth is guaranteed by Theorem \ref{born}, but, it is non trivial to explicit.
\end{remark}

\begin{proposition}\label{p5}
Assume $\mu\neq0$ and $\alpha\in (0,1).$ Then, the affine cusps, presented in Proposition \ref{regularidade_do_envelope}, are ordinary affine cusps if and only if  $\alpha\mu_{ss}\neq (1-\alpha)\mu_s$, where the derivatives are evaluated at the affine cusp point.
\end{proposition}

\begin{proof}
The condition for an ordinary affine cusp is in fact that the second and third derivatives of $\gamma$ evaluated at the affine cusp point should be independent vectors. We see that $X_s=A(s)v^\alpha,$ where $v^\alpha=((1-\alpha)\gamma_s+\alpha\gamma_{ss})$. By differentiating $X_s$ twice, we obtain $$ X_{ss}=A_sv^\alpha+Av^\alpha_s, \ \ X_{sss}=A_{ss}v^\alpha+2A_sv^\alpha_s+Av^\alpha_{ss}.$$
If $A(s)=0$, we have $$[X_{ss},X_{sss}]=2A_s^2[v^\alpha,v^\alpha_s]=2A_s^2\left((1-\alpha)^2+\mu\alpha^2\right).$$ Thus, the condition for these vectors to be linearly dependent is $2A_s^2[v^\alpha,v^\alpha_s]=0$, e. g., $A_s=0$. Thus, we obtain the required formula, since $(1-\alpha)^2+\mu\alpha^2\neq0$, already that $\mu\neq0$ and $\alpha\in(0,1)$.
\end{proof}

This article aims to study not only a single value of $\alpha$ but also what happens to $E_\alpha$ as $\alpha$
varies. For such, the investigation is conducted in a broader context.


\section{Discriminants and singularity theory}

Consider the family  of functions of one variable $s$ with three parameters $(x,y,\alpha)$ 

\begin{equation}\label{F}
F(X,\alpha,s)=\left[ X-\gamma(s), (1-\alpha)\gamma_s+\alpha\gamma_{ss}\right],
\end{equation}

 where $X=(x,y), \gamma(s)=(x(s),y(s))$ and $s$ is the parameter of affine arclength. The discriminant of this family is given by 

\begin{equation}\label{discriminant}
D_F=\{(X,\alpha) : \textrm{there is} \  s \ \textrm{such that} \ F(X,\alpha,s)=F_s(X,\alpha,s)=0\}.
\end{equation}
This discriminant is the union of all the envelopes of lines $L_\alpha$ for each $\alpha.$

Now, consider the discriminant $D_F$ and the function $h(x,y)=\alpha$. The level sets $h=constant$ are the individual envelopes of the family. We intend to investigate precisely how they change as $\alpha$ varies.

\begin{example} Consider the ellipse $\gamma(t)=(3\cos(t),2\sin(t))$ and the curve $\sigma(t)=(\cos(2t)-\cos(t+1.9), $ $\sin(2t)+\sin(t))$. Let $D_F$ be the discriminant surface associated to $\gamma$ and $D_G$, the discriminant surface associated to $\sigma$, as illustrated in Fig. \ref{fig3}. Remark that, in discriminant surface, $D_G$ seems to have cuspidal edges and swallowtail surfaces\footnote{For details about cusps, cuspidal edges and swallowtail surface, see \cite{Giblin4}.}, and the function $h$ seems to have level sets which undergo a swallowtail transition for certain values of $\alpha$. We are interested in verifying these observations.
\end{example}

\begin{figure}[ht]
\centering
\subfigure[Discriminant surface $D_F$]{\includegraphics[width=4cm]{{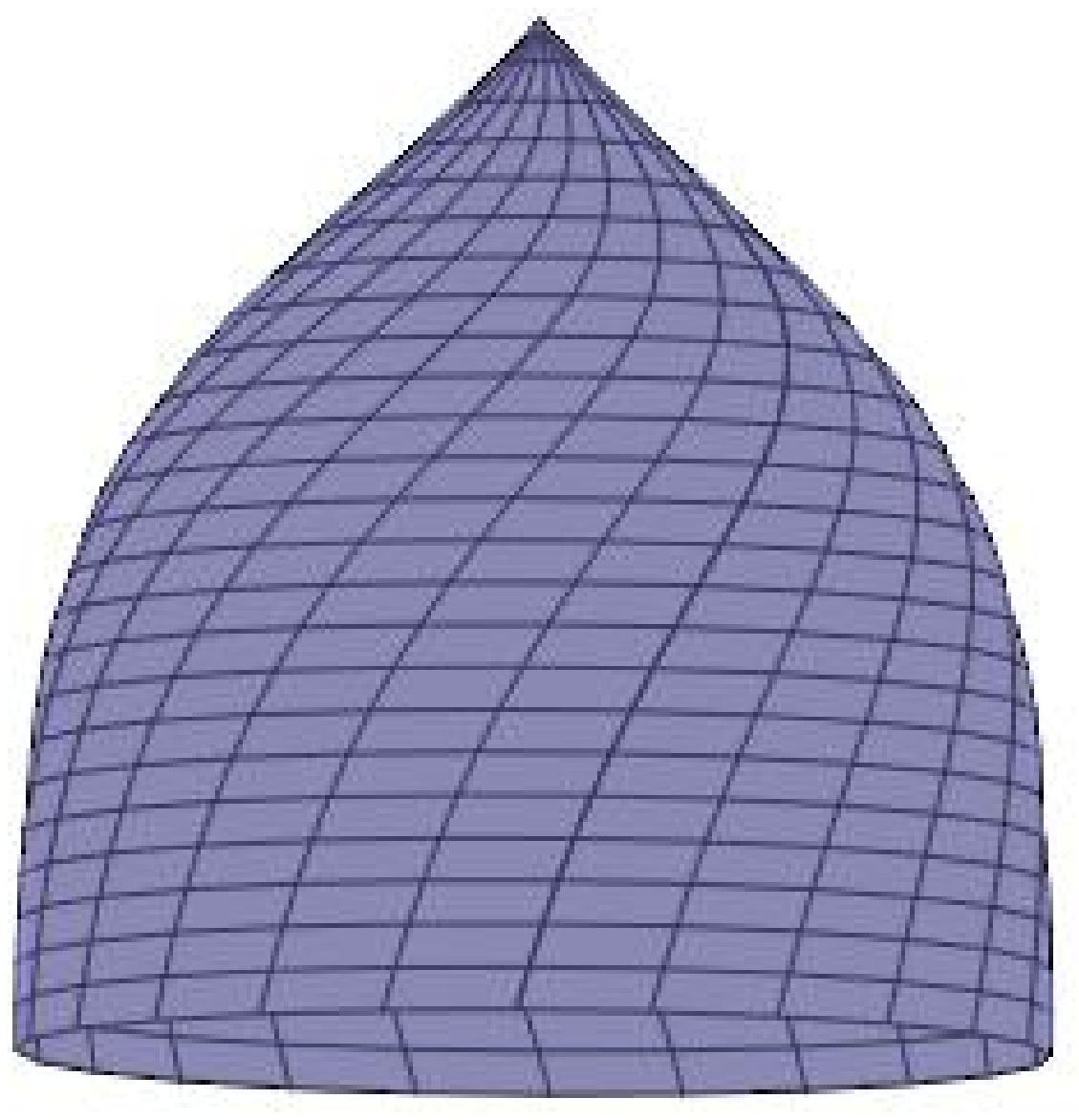}}}
\subfigure[Discriminant surface $D_G$]{\includegraphics[width=4cm]{{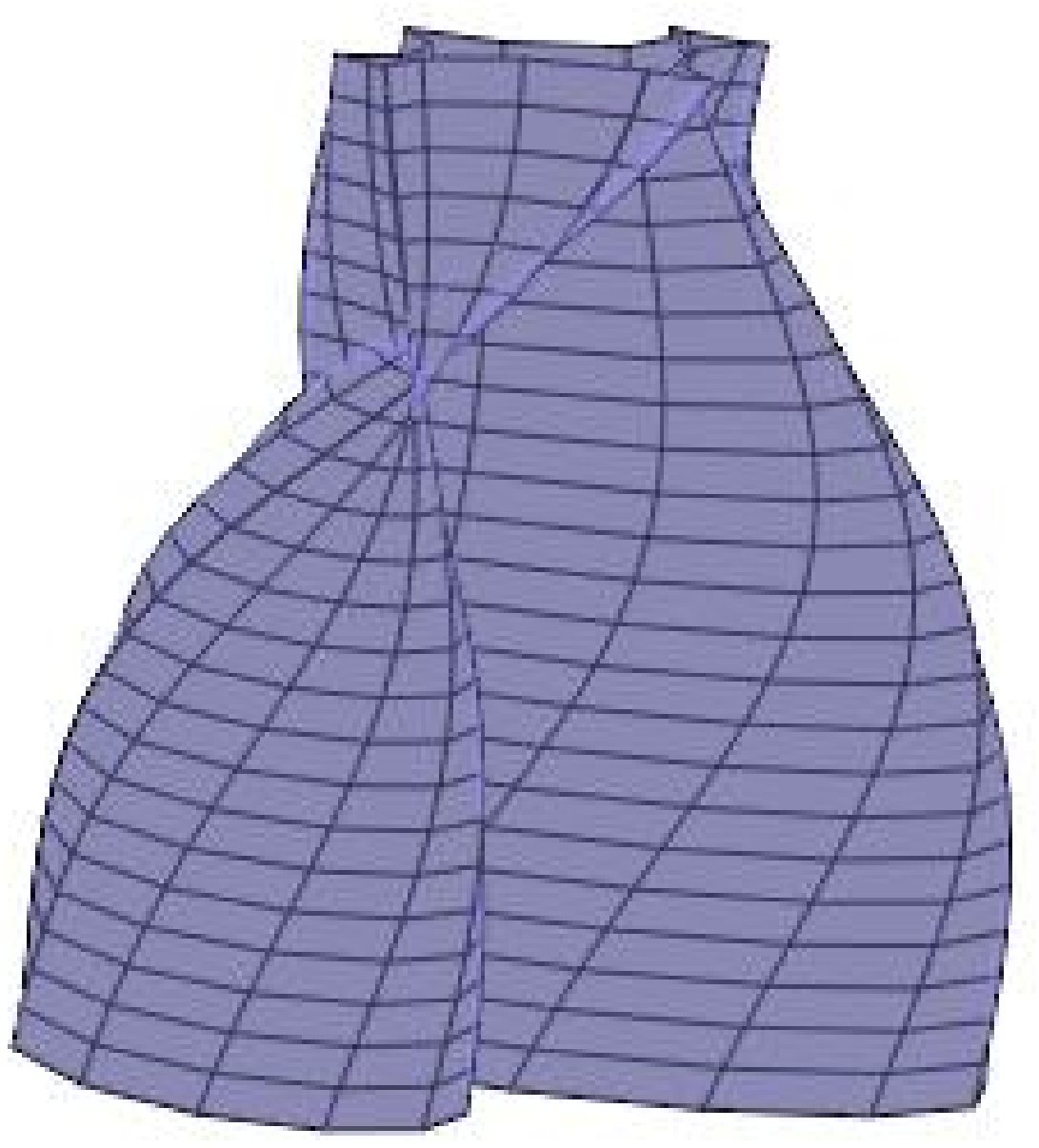}}}
\caption{For the discriminant surface $D_F$ ($\alpha$-axis is vertical) $\alpha=0$ corresponds to the bottom that is the original ellipse and $\alpha=1$ is the top, which corresponds to the envelope of affine normals, which degenerates at a point. For the discriminant surface $D_G$, $\alpha=0$ corresponds to original curve $\sigma$, and $\alpha=1$ corresponds to affine evolute (envelope of affine normals, which has six cusps). In $D_G$, cuspidal edges appear and the horizontal sections seem to undergo a swallowtail transition.} \label{fig3}
\end{figure}

To verify the observations given in the example above, we shall apply the results from the singularity theory which allow us to make precise statements about how the envelopes evolve as $\alpha$ changes.

%

\begin{definition}
For $(X_0,\alpha_0)=(x_0,y_0,\alpha_0)$ the function $f(s) = F(X_0,\alpha_0,s)$ has singularity
\begin{enumerate}
\item[(i)] type $A_2$ at $s = s_0$ if $f'(s_0) = f''(s_0) = 0, f'''(s_0)\neq 0$, \\
\item[(ii)] type $A_3$ at $s = s_0$ if $f'(s_0) = f''(s_0) = 0, f'''(s_0)=0, f^{(4)}(s_0)\neq 0$.
\end{enumerate}
\end{definition}

\begin{proposition}\label{P5}
Let the point $(x_0,y_0,\alpha_0)=(X_0,\alpha_0)$, which satisfies $F=F_s=0$ and suppose $\mu(s_0)\neq 0.$ Then, $f(s)=F(X_0,\alpha_0,s)$ has singularity
\begin{enumerate}
\item[$($i$)$] type $A_2$ at $s_0$, if $ \alpha\mu_{ss}-(1-\alpha)\mu_s\neq0;$ \\
\item[$($ii$)$] type $A_3$ at $s_0$, if $\alpha^5\mu_{sss}\neq-(1-\alpha)^3((1-\alpha)^2+\alpha^2\mu).$
\end{enumerate}
\end{proposition}

\begin{proof} 
$(i)$  The equation $F_{ss}=0$ implies 
\begin{equation} \label{eq12}
\alpha^3\mu_s=(1-\alpha)((1-\alpha)^2+\alpha^2\mu).
\end{equation}

By differentiating $F$ with respect to $s$ three times, using the equation \eqref{eq12} and the hypothesis $\alpha\mu_{ss}-(1-\alpha)\mu_s\neq0$, we obtain $F_{sss}\neq0.$

$(ii)$  The equation $F_{sss}=0$ implies 
\begin{equation} \label{eq13}
\alpha\mu_{ss}-(1-\alpha)\mu_s=0.
\end{equation}

By differentiating $F$ with respect to $s$ four times, using the equation \eqref{eq13} and the hypothesis $\alpha^5\mu_{sss}\neq-(1-\alpha)^3((1-\alpha)^2+\alpha^2\mu),$ we obtain $F_{ssss}\neq0.$
\end{proof}

We highlight the following criterion (for further details, see \cite{Giblin4}), which is used for studying the behavior of singularities. 

\begin{definition}[Criterion for versality]\label{versal} Let $H(X,z,s)=H(x,y,z,s)$ be a family to $3$ parameters. Suppose that $H = H_s = 0$ at $(X_0,z_0, s_0)$ and $h(s) = H(X_0,z_0,s)$ has an $A_r$ singularity at $s_0.$ Consider the partial derivatives $H_{x} ,H_y, H_z$, evaluated at $(X_0,z_0,s_0)$ and, in particular, their Taylor polynomials $T_i$ up to degree $r-1,$ expanded about $s_0$ (so these have $r$ terms). The family $H(X,z,s)$ is called a versal unfolding of $h$ at $s_0$ if the $T_i$ spans a vector space of dimension $r$. Thus, if the coefficients in the $T_i$ are placed as the columns of an $r \times 3$ matrix, the rank is $r.$ Clearly, this is possible only for $r \leq 3.$
\end{definition}

\begin{remark}[See \cite{Giblin4}]\label{T3}
It is known about the Singularity Theory that, if a family $H$ satisfies the criterion of the definition above, then in a neighborhood of $(X_0,z_0)\in D_H$, the discriminant is locally diffeomorphic to a cuspidal edge surface when $r=2$, and to a swallowtail surface, when $r=3$.
\end{remark}

In the next result, we showed that $F$, defined by equation \eqref{F}, satisfies the criterion given in Definition \ref{versal}.

\begin{theorem}
The family $F$ satisfies the conditions of the criterion given in Definition \ref{versal}. Thus, when $f(s) = F(X_0,\alpha_0, s)$ has an $A_r$ singularity at $s_0, r = 2$ or $3$, in the cases covered by Proposition \ref{P5}, the discriminant $D_F$ is always locally diffeomorphic to a standard cuspidal edge ($r = 2$) or a standard swallowtail surface ($r = 3$) in a neighborhood of $(X_0,\alpha_0)$.
\end{theorem}

\begin{proof}
Recall that 
$$
\begin{array}{ccl}
F(X,\alpha,s) & = & [X-\gamma,(1-\alpha)\gamma_s+\alpha\gamma_{ss}] \\
 & = & (x-x(s))((1-\alpha)y_s+\alpha y_{ss})-(y-y(s))((1-\alpha)x_s+\alpha x_{ss}).
\end{array}$$

We shall prove that $F$ is versal, e. g., that the matrix $J$  bellow has rank $2$. $$J=\left(\begin{array}{ccc}
F_x & F_y & F_\alpha \\
F_{xs} & F_{ys} & F_{\alpha s} \\
\end{array}\right)=$$$$=\left(\begin{array}{ccc}
(1-\alpha)y_s+\alpha y_{ss} & -(1-\alpha)x_s-\alpha x_{ss} &  \dfrac{\alpha}{(1-\alpha)^2+\mu\alpha^2} \\
(1-\alpha)y_{ss}-\alpha \mu y_{s} & -(1-\alpha)x_{ss}+\alpha \mu x_{s} & -\dfrac{1-\alpha}{(1-\alpha)^2+\mu\alpha^2} \\
\end{array}\right).
$$

Observe that the determinant of $J1$ is $(1-\alpha)^2+\mu\alpha^2$, where $$J1=\left(\begin{array}{cc}
F_x & F_y \\
F_{xs} & F_{ys} \\
\end{array}\right),$$ which is nonzero by assumption (see Remark \ref{r1} (a)). This finalizes the case where the singularity is type $A_2$ at $s_0.$

In the case where the singularity is type $A_3$ at $s_0$ we shall prove that the matrix $\bar{J}$ bellow has rank $3.$ to show that $F$ is versal.

$$\bar{J}=\left(\begin{array}{ccc}
F_x & F_y & F_\alpha \\
F_{xs} & F_{ys} & F_{\alpha s} \\
F_{xss} & F_{yss} & F_{\alpha ss} \\
\end{array}\right)=$$$$=\dfrac{1}{(1-\alpha)^2+\mu\alpha^2}\left(\begin{array}{ccc}
(1-\alpha)y_s+\alpha y_{ss} & -(1-\alpha)x_s-\alpha x_{ss} & \alpha \\
(1-\alpha)y_{ss}-\alpha \mu y_{s} & -(1-\alpha)x_{ss}+\alpha \mu x_{s} & -(1-\alpha) \\
F_{xss} & F_{yss} & \dfrac{(1-\alpha)^2}{\alpha} \\
\end{array}\right),
$$
where the derivatives $F_{xss}$ and $F_{yss}$ are given by
$$F_{xss}=-(1-\alpha)^3y_s-\alpha^3y_{ss}-2\alpha^2(1-\alpha)\mu y_s$$
$$F_{yss}= (1-\alpha)^3x_s+\alpha^3x_{ss}+2\alpha^2(1-\alpha)\mu x_s.$$ Observe that, using  the equation \eqref{eq12}, the determinant of $\bar{J}$ is given by $(\alpha^2\mu+3(1-\alpha)^2)/\alpha$, which is equal to zero iff $\mu(s_0)=-3(1-\alpha)^2/\alpha^2<0.$ If $\mu(s)$, for $s\in I$, does not change the signal, then $\gamma$ is not closed, which is contradiction. Now, if $\mu(s)$ changes the signal, then for some $\bar{s_0} \in I$, we have $\mu(\bar{s_0})=0$. Therefore, $\gamma$ has av affine inflexion at this point, which is another contradiction.
\end{proof}

The next result presents properties of the behavior of discriminant surface. For details, see \cite{arnold,bruce}.

\begin{proposition}
\begin{itemize}\label{swal}
\item[(i)] For a cuspidal edge surface, with $(X_0,\alpha_0)$ on the line of cusps, the level sets of $h$ on $D_F$ will all be cusped curves, provided that the plane $K$ does not contain the tangent to the line of cusps through $(X_0,\alpha_0).$ (“$K$ is transverse to the line of cusps”.)
On the other hand, the level sets undergo a “beaks” or “lips” transition, provided that $K$ does contain this tangent but does not coincide with the limiting tangent plane to the cuspidal edge surface at points approaching $(X_0,\alpha_0).$ (“$K$ is transverse to this limiting tangent plane”.)

\item[(ii)] For a swallowtail surface, with $(X_0,\alpha_0)$ at the swallowtail point, the level sets on $D_F$ undergo a swallowtail transition, with two cusps merging and disappearing, provided that $K$ does not contain the limiting tangent to the lines of cusps on $D_F$ at $(X_0,\alpha_0)$. (“$K$ is transverse to this limiting tangent line”.)
\end{itemize}
\end{proposition}

In the next result, we prove that the Proposition \ref{swal} is always satisfied for the discriminant $D_F$.

\begin{theorem}
The affine evolutoids $E_\alpha$ evolve locally according to a stable cusp at $A_2$ points, where the affine curvature $\mu\neq0$; according to a swallowtail transition at $A_3$ points, where $\mu\neq0$. At all other points, the envelope $E_\alpha$ is a smooth curve.
\end{theorem}

\begin{proof}
Consider that $f(s)=F(X_0,\alpha_0,s)$ has an $A_2$ or $A_3$ singularity in $s=s_0$. In the case $A_2$, we known that $D_F$ is locally diffeomorphic to a cuspidal edge surface close to $(X_0,\alpha_0).$ This cuspidal edge is given by $F=F_s=F_{ss}=0$, e.g., three equations in the four variables $x, y, \alpha, s,$ and the solutions are then projected to $(x,y,\alpha)-$space, where $D_F$ lies. Let $J_2$ be the matrix formed by three columns of $\bar{J}$, thus increasing the column formed by $F_s, F_{ss}$ and $F_{sss}$. Notice that the fourth column is  the transpose of vector $(0,0,F_{sss}\neq0)$ at an $A_2$ point, and the transpose of $(0,0,0)$ at an $A_3$ point.  Thus, for an $A_2$ point, we can always find a kernel vector $(\bar{x},\bar{y},\bar{\alpha},\bar{s})$ whose first three components are not all zero, using the first two rows of $J_2$, and then determine $\bar{s}$ using the third row of $J_2$, since $F_{sss}\neq0.$ Then, $(\bar{x},\bar{y},\bar{\alpha})$ is a nonzero tangent vector to the line of cusps $C$ on $D_F$ in $(x,y,\alpha)-$space. However, this cannot be done with $\bar{\alpha}=0$ in view of the non-singularity of $J_1$, sub-matrix of $J_2$. So, a tangent vector to $C$ will never be horizontal, and changing $\alpha$ to nearby values gives a stable cusp.
There is clearly a problem with this argument at an $A_3$ point $(X_0,\alpha_0)$, where $F_{sss}=0$, since in view of the non-singularity of $\bar{J}$, all kernell vectors of $J_2$ have the form $(0,0,0,\bar{s})$. This simply says that, in $(x,y,\alpha)-$space, the curve $C$ on $D_F$ is singular at $(x_0,y_0,\alpha_0)$, which is true, since $D_F$ is a swallowtail surface and the space curve $C$ itself has a cusp at $(x_0,y_0,\alpha_0)$. However, the above argument still applies, by taking a unit tangent vector $(\bar{x},\bar{y},\bar{\alpha})$ and moving towards $(x_0,y_0,\alpha_0)$ along $C$: the last component cannot tend to $0$ without the other two tending to $0$ as well, which is a contradiction. In the present case, we can be more explicit: a tangent vector to $C$, obtained from the first two rows of $J_2$ is $$\left(2\alpha(\alpha-1)x_{ss}+x_s(\alpha^2\mu-(1-\alpha)^2),2\alpha(\alpha-1)y_{ss}+y_s(\alpha^2\mu-(1-\alpha)^2),-((1-\alpha)^2+\alpha^2\mu)^2\right).$$ It is clear that this cannot have a limit in which the third component is $0.$

\end{proof}

\bibliographystyle{amsplain}

\end{document}